\documentclass[12pt]{article}

\RequirePackage{fix-cm}
\usepackage{amsmath, amsfonts, amssymb}
\usepackage{verbatim}
\usepackage[dvipdfm,a4paper,left=20mm,right=20mm,top=35mm,bottom=35mm]{geometry}
\usepackage{color}
\usepackage{bm}
\usepackage{graphicx}
\usepackage{framed}
\usepackage{hyperref}
\usepackage[francais,english]{babel}
\usepackage{accents}

\def\english{\selectlanguage{english}}

\newtheorem{defin}{Definition}[section]
\newtheorem{theorem}{Theorem}[section]
\newtheorem{proposition}{Proposition}[section]

\newtheorem{lemma}{Lemma}[section]

\newtheorem{remark}{Remark}[section]

\newcounter{steps}
\newenvironment{proof}[1][]{%
\par\medbreak\setcounter{steps}{0}
{\noindent\bfseries Proof#1. }} {\hfill\fbox{\ }\medbreak}

\newcounter{substeps}[steps]

\def\eps{\varepsilon}
\def\R{\mathbb{R}}
\def\N{\mathbb{N}}
\def\O{\mathcal{O}}

\def\({\begin{eqnarray}}
\def\){\end{eqnarray}}
\def\[{\begin{eqnarray*}}
\def\]{\end{eqnarray*}}
\def\part#1#2{{\partial #1\over\partial #2}}

\def\dx{\nabla_x} 
\def\dv{\, \nabla_v}
\def\dt{\partial_t}
 
\def\d{\, \text{d}}

\newcommand{\intv}[1]{
\int _{\R ^d} \!#1 \;\mathrm{d}v}

\newcommand{\intvp}[1]{
\int _{\R ^d} \!#1 \;\mathrm{d}v^\prime}

\newcommand{\vp}[0]{
v^{\prime}}

\newcommand{\fe}[0]{
f ^\varepsilon}

\newcommand{\fin}[0]{
f ^{\mathrm{in}}}

\newcommand{\Divx}[0]{
\mathrm{div}_x}

\newcommand{\Divv}[0]{
\mathrm{div}_v}

\newcommand{\fz}[0]{
f}

\newcommand{\fo}[0]{
f ^{1}}

\newcommand{\lime}[0]{
\lim _{\varepsilon \searrow 0}}

\newcommand{\mo}[0]{
M_{\Omega}}

\newcommand{\mof}[0]{
M_{\Omega [\fz]}}

\newcommand{\calmo}[0]{
{\mathcal M}_{\Omega}}

\newcommand{\zo}[0]{
Z_\Omega}

\newcommand{\Phio}{
\Phi _\Omega}

\newcommand{\Phiof}{
\Phi _{\Omega[\fz]}}

\newcommand{\sphere}[0]{
\mathbb{S}^{d-1}}

\newcommand{\ltmo}{
L^2 _{\mo}}

\newcommand{\homo}{
H^1 _{\mo}}

\newcommand{\thomo}{
\tilde{H}^1 _{\mo}}

\newcommand{\intth}[1]{
\int _0 ^{\pi} #1 \;\mathrm{d}\theta}

\newcommand{\calL}[0]{
\mathcal{L}}

\newcommand{\pf}[0]{
P_f}

\newcommand{\sumi}[0]{
\sum _{i = 1}^{d-1}}

\newcommand{\sumj}[0]{
\sum _{j = 1}^{d-1}}

\def\pmb#1{\setbox0=\hbox{$#1$}
  \kern-.025em\copy0\kern-\wd0
  \kern-.05em\copy0\kern-\wd0
  \kern-.025em\raise.0433em\box0 }

\begin{document}
\english

\title{Hydrodynamic limits for kinetic flocking models of Cucker-Smale type}


\author{
P. Aceves-S\'anchez
\thanks{Department of Mathematics, Imperial College London, London SW7 2AZ UK.  E-mail : {\tt p.aceves-sanchez@imperial.ac.uk}
},
M. Bostan
\thanks{Aix Marseille Universit\'e, CNRS, Centrale Marseille, Institut de Math\'ematiques de Marseille, UMR 7373, Ch\^ateau Gombert 39 rue F. Joliot Curie, 13453 Marseille Cedex 13 FRANCE. E-mail : {\tt mihai.bostan@univ-amu.fr}
}, 
J.A. Carrillo
\thanks{Department of Mathematics, Imperial College London, London SW7 2AZ UK.  E-mail : {\tt carrillo@imperial.ac.uk}
},
P. Degond
\thanks{Department of Mathematics, Imperial College London, London SW7 2AZ UK.  E-mail : {\tt p.degond@imperial.ac.uk}.
}
}
\date{ (\today)}

\maketitle

%


\begin{abstract}
We analyse the asymptotic behavior for kinetic models describing the collective behavior of animal populations. We focus on models for self-propelled individuals, whose velocity relaxes toward the mean orientation of the neighbors. The self-propelling and friction forces together with the alignment and the noise are interpreted as a collision/interaction mechanism acting with equal strength. We show that the set of generalized collision invariants, introduced in \cite{DM08}, is equivalent in our setting to the more classical notion of collision invariants, i.e., the kernel of a suitably linearized collision operator. After identifying these collision invariants, we derive the fluid model, by appealing to the balances for the particle concentration and orientation. We investigate 
the main properties of the macroscopic model for a general potential with radial symmetry. 
\end{abstract}








\section{Introduction}
\label{Intro}
Kinetic models have been introduced in the last years for the mesoscopic description of collective behavior of agents/particles with applications in collective behavior of cell and animal populations, see \cite{review,review2,MT14} and the references therein for a general overview on this active field. These models usually include alignment, attraction and repulsion as basic bricks of interactions between individuals. We refer to \cite{Neu77,BraHep77,Dob79,CCR10,BCC11,BCC12,HLL09,CCH,CCHS} and the references therein for a derivation of the kinetic equation from the microscopic models. 

In this work, we focus on the derivation of macroscopic equations for the collective motion of self-propelled particles with alignment and noise when a cruise speed for individuals is imposed asymptotically for large times as in \cite{DorChuBerCha06, ChuHuaDorBer07, ChuDorMarBerCha07, CarDorPan09, UAB25}. More precisely, in the presence of friction and self-propulsion and the absence of other interactions, individuals/particles accelerate or break to achieve a cruise speed exponentially fast in time. The alignment between particles is imposed via localized versions of the Cucker-Smale or Motsch-Tadmor reorientation procedure \cite{CS2,HT08,HL08,CFRT10,review,MT11} leading to relaxation terms to the mean velocity modulated or not by the density of particles. By scaling the relaxation time towards the asymptotic cruise speed, or equivalently, penalizing the balance between friction and self-propulsion, this alignment interaction leads asymptotically to variations of the classical kinetic Vicsek-Fokker-Planck equation with velocities on the sphere, see \cite{VicCziBenCohSho95,DM08,FL11,DFL10,DFL15,BosCar13,BosCar15}. It was shown in \cite{BCCD16} that particular versions of the localized kinetic Cucker-Smale model can lead to phase transitions driven by noise. Moreover, these phase transitions are numerically stable in this asymptotic limit converging towards the phase transitions of the limiting versions of the corresponding kinetic Vicsek-Fokker-Planck equation.

In this work, we choose a localized and normalized version of the Cucker-Smale model not showing phase transition. More precisely, let us denote by $f = f(t,x,v) \geq 0$ the particle density in the phase space $(x,v) \in \R^d \times \R^d$, with $d \geq 2$. The standard self-propulsion/friction mechanism leading to the cruise speed of the particles in the absence of alignment is  given by the term $\mathrm{div}_v \{f (\alpha - \beta |v|^2)v\}$ with $\alpha, \beta >0$, and the relaxation toward the normalized mean velocity writes $\mathrm{div}_v \{ f( v - \Omega[f])\}$ cf. \cite{DM08,VicCziBenCohSho95,CouKraFraLev05}. 
Here, for any particle density $f(x,v)$, the notation $\Omega[f]$ stands for the orientation of the mean velocity
\begin{eqnarray*}
\Omega[f] := \left \{
\begin{array}{ll}
\displaystyle \frac{\intv{f(\cdot,v)v}}{\left |\intv{f(\cdot,v)v}\right |}, & \mbox{ if } \intv{f(\cdot,v)v} \in \R^d \setminus \{0\} \, ,\\
0, & \mbox{ if } \intv{f(\cdot,v)v} = 0 \, .
\end{array}
\right.
\end{eqnarray*}
Notice that we always have
\[
\rho u[f] := \intv{f(\cdot,v)v} = \left | \intv{f(\cdot,v)v} \right | \Omega [f] \qquad\mbox{with } \rho:=  \intv{f(\cdot,v)}.
\]
Let us remark that the standard localized Cucker-Smale model would lead to $\rho \mathrm{div}_v \{ f( v -u[f])\}$ while the localized Motsch-Tadmor model would lead to  $\mathrm{div}_v \{ f( v -u[f])\}$. Our relaxation term towards the normalized local velocity $\Omega [f]$ does not give rise to phase transition in the homogeneous setting on the limiting Vicsek-Fokker-Planck-type model on the sphere according to \cite{DFL10} and it produces a competition to the cruise speed term comprising a tendency towards unit speed. Including random Brownian fluctuations in the velocity variable leads to the kinetic Fokker-Planck type equation
\[
\partial _t f + v \cdot \nabla _x f + \mathrm{div}_v \{f ( \alpha - \beta |v|^2)v\} = \Divv \{ \sigma \nabla _v f + f ( v - \Omega[f]) \},\;\;(t,x,v) \in \R_+ \times \R^d \times \R^d \,.
\]
We include this equation in a more general family of equations written in a compact form as
\begin{equation}\label{eq:kinetic}
\dt f + v \cdot \dx f = Q (f) \, ,
\end{equation}
where 
\begin{equation}\label{eq:col}
 Q(f) = \Divv \{ \sigma \dv f + f(v - \Omega[f]) + \eta f \dv V \} \, ,
\end{equation}
for any density distribution $f$ with $V$ a general confining potential in the velocity variables and $\eta > 0$ (see Lemma \ref{RSVol4} for more information on the type of potentials that we consider). In the particular example considered above we take $V=V_{\alpha, \beta} (|v|):= \beta \frac{|v|^4}{4} - \alpha \frac{|v|^2}{2}\;$.

We investigate the large time and space scale regimes of the kinetic tranport equation \eqref{eq:kinetic} with collision operator given by \eqref{eq:col}. Namely, we study the asymptotic behavior when $\eps \to 0$ of 
\begin{equation}
\label{eq:main}
\dt \fe + v \cdot \dx \fe = \frac{1}{\eps} Q (\fe) \, ,
\end{equation}
supplemented with the initial condition
\begin{equation*}
\label{eq:IC}
\fe (0,x,v) = \fin (x,v),\;\;(x,v) \in \R^d \times \R^d \, .
\end{equation*}

The rescaling taken in the kinetic transport equation \eqref{eq:main} with confining potential $V_{ \alpha, \beta}$ can be seen as an intermediate scaling between the ones proposed in \cite{BosCar15} and \cite{BD14}. The difference being that we have a relaxation towards the normalized mean velocity $\Omega[f]$ rather than the mean velocity $u[f]$ as in \cite{BD14,BosCar15}. This difference is important since in the first case there is no phase transition in the homogeneous limiting setting on the sphere as we mentioned above, while in the second there is, see \cite{DFL10,BCCD16,BosCar15}. In fact, in \cite{BosCar15} the scaling corresponds to $\eta = 1 / \eps$ in \eqref{eq:main}, that is the relaxation to the cruise speed is penalized with a term of the order of $1 / \eps^2$. Whereas in \cite{BD14} the scaling correponds to $\eta = \eps$, that is the cruise speed is not penalized at all.

The methodology followed in \cite{BosCar15} lies within the context of measure solutions by introducing a projection operator onto the set of measures supported in the sphere whose  radius is the critical speed $r =\sqrt{\alpha/\beta}$. These technicalities are needed because the zeroth order expansion of $f_\eps$ lives on the sphere. This construction followed closely the average method in gyro-kinetic theory \cite{BosAsyAna,BosTraEquSin,BosGuiCen3D}.

However, in our present case we will show in contrast to \cite{BosCar15,BD14} that there are no phase transitions which is in accordance with the results obtained in \cite{DFL15} for the kinetic Vicsek-Fokker-Planck equation with analogous alignment operator on the sphere.
A modified version of \eqref{eq:kinetic}-\eqref{eq:col} in which phase-transitions occur was studied in \cite{BCCD16} whose analysis is postponed to a future work to focus here on the mathematical difficulties of the asymptotic analysis. Another difference in the present case is that the zeroth order expansion of $f_\eps$ will be parameterized by Von Mises-Fisher distributions in the whole velocity space, that is $f(t,x,v) = \rho (t,x) M_{\Omega(t,x)}(v)$, with $\rho$ and $\Omega$ being, respectively, the density and the mean orientation of the particles. And where for any $\Omega \in \mathbb{S}^{d-1}$ we define (see Sect. 2)
\begin{equation}\label{eq:M_omega}
 M_{\Omega} ( v) = \frac{ 1}{Z_\Omega} \exp \bigg( - \frac{ \Phi_{ \Omega} ( v)}{ \sigma} \bigg) \, , \; \text{ with } Z_\Omega = \int_{ \R^d} \exp \bigg( - \frac{ \Phi_{ \Omega} ( v')}{ \sigma} \bigg) \d v'
\end{equation}
and
\[
\Phi ( v) = \frac{ | v - \Omega|^2}{ 2} + V( | v|).
\]

The main result of this paper is the asymptotic analysis of the singularly perturbed kinetic transport 
equation of Cucker-Smale type \eqref{eq:main}. The particle density $\rho$ and the orientation $\Omega$
obey the hydrodynamic type equations given in the following result.

\begin{theorem}
\label{MainResult1}
Let $\fin \geq 0$ be a smooth initial particle density with nonvanishing orientation at any $x \in \R^d$. For any $\eps >0$ we consider the problem
\[
\partial _t \fe + v \cdot \nabla _x \fe = \frac{1}{\eps} \Divv \{ \sigma \nabla _v \fe + \fe ( v - \Omega[\fe] ) + \fe \nabla _v V(|v|)\},\;\;(t,x,v) \in \R_+ \times \R^d \times \R^d \, ,
\]
with initial condition
\[
\fe (0,x,v) = \fin (x,v),\;\;(x,v) \in \R^d \times \R^d \, .
\]
At any $(t,x) \in \R_+ \times \R^d$ the leading order term in the Hilbert expansion $f_\eps = f + \eps f_1 + \ldots$ is an equilibrium distribution of $Q$, that is $f(t,x,v) = \rho (t,x) M_{\Omega(t,x)}(v)$ with $M_{\Omega(t,x)}(v)$ defined in \eqref{eq:M_omega}, where the concentration $\rho$ and the orientation $\Omega$ satisfy 

\begin{equation}\label{eq:bas_rho}
 \partial _t \rho + \Divx (\rho c_1 \Omega) = 0,\;\;(t,x) \in \R_+ \times \R^d \, ,
\end{equation}
\begin{equation}\label{eq:bas_omega}
 \partial _t \Omega + c_2(\Omega \cdot \nabla _x ) \Omega + \sigma (I_d - \Omega \otimes \Omega) \frac{\nabla _x \rho }{\rho} = 0,\;\;(t,x) \in \R_+ \times \R^d \, ,
\end{equation}
with initial conditions
\[
\rho (0,x) = \intv{\fin (x,v)},\;\;\Omega (0,x) = \frac{\intv{\fin (x,v)v}}{\left | \intv{\fin (x,v)v} \right |},\;\;x\in \R^d \, . 
\]
The constants $c_1, c_2$ are given by
\begin{equation*}
 c_1 = \frac{\int_{\R_+}r^d \intth{\cos \theta \, e(\cos \theta, r) \sin ^{d-2} \theta}\d r}{\int_{\R_+}r^{d-1} \intth{e(\cos \theta, r) \sin ^{d-2} \theta}\d r} \, ,
\end{equation*}
\begin{equation*}
c_2 = \frac{\int_{\R_+}r^{d+1} \intth{\cos \theta \, \chi (\cos \theta, r) \,  e(\cos \theta, r) \sin ^{d-1} \theta}\d r}{\int_{\R_+}r^{d} \intth{\chi (\cos \theta, r) \, e(\cos \theta, r) \sin ^{d-1} \theta}\d r} \, ,
\end{equation*}
and the function $\chi$ solves
\begin{align*}
\label{EllipticChi}
- \sigma \partial _c \left [r^{d-3} (1 - c^2) ^{\frac{d-1}{2}} e(c,r) \partial _c \chi   \right ] & - \sigma \partial _r \left [r^{d-1}(1 - c^2 ) ^{\frac{d-3}{2}} e(c,r) \partial _r \chi  \right ] + \sigma (d-2) r ^{d-3} (1- c^2) ^{\frac{d-5}{2}} e \chi \nonumber \\
& = r^d (1-c^2) ^{\frac{d-2}{2}}e(c,r) \, ,
\end{align*}
where $e(c,r) = \exp(rc/\sigma) \exp(- (r^2 + 1)/(2\sigma) - V(r)/\sigma)$.
\end{theorem}

Our article is organized as follows. First, in Section \ref{sec:preliminaries} we state auxiliary results allowing us to discuss the kernel of the collision operator. Then in Section \ref{CharCollInv} we concentrate on the characterization of the 
collision invariants. We prove that the generalized collision invariants introduced in \cite{DM08} coincide with the kernel of a suitable linearised collision operator. We explicitly describe the collision invariants in Section \ref{IdentCollInv} and investigate their symmetries. 
Finally, the limit fluid model is determined in Section \ref{Hydro} and we analyse its main properties. 

\section{Preliminaries}\label{sec:preliminaries}

Plugging into \eqref{eq:main} the Hilbert expansion 
\[
\fe = \fz + \eps \fo  + \ldots \, ,
\]
we obtain at the leading order 
\begin{equation}
\label{eq:0exp}
Q(\fz) = 0 \, ,
\end{equation}
whereas to the next order we get
\begin{equation}
\label{eq:1exp}
\partial_t \fz + v \cdot \nabla_x \fz = \lime \frac{1}{\eps}\{Q(\fe) - Q(\fz)\}= \d Q_f (\fo) =: \calL _f ( \fo) \, ,
\end{equation} 
where $\d Q_f$ denotes the first variation of $Q$ with respect to $f$. The constraint \eqref{eq:0exp} leads immediately to the equilibrium
\[
\mo (v) = \frac{1}{\zo}\exp \left (- \frac{\Phio(v)}{\sigma}  
\right ) \, , \quad \text{  with } \zo = \intvp{\exp \left ( - \frac{\Phio(\vp)}{\sigma}  
\right )} \, ,
\]
where
\begin{equation}\label{eq:pot}
 \Phio (v) = \frac{|v - \Omega |^2}{2} + V(|v|) \, .
\end{equation}

Indeed, by using the identity
\begin{equation}\label{eq:grad_m}
\nabla _v \mo = - \frac{\mo (v)}{\sigma} \nabla _v \Phio = - \frac{\mo (v)}{\sigma} (v - \Omega + \nabla _v V(|v|)\;) \, ,
\end{equation}
we can recast the operator $Q$ as
\begin{align*}
\label{eq:Qiden}
Q(\fz) & = \Divv \left ( \sigma \nabla _v \fz + \fz \nabla _v \Phiof \right ) = \sigma \Divv \left [\mof \nabla _v \left ( \frac{\fz}{\mof} \right ) \right ] \, .
\end{align*}
We denote by $\sphere$ the set of unit vectors in $\R^d$. For any $\Omega \in \sphere$, we consider the weighted spaces
\[
\ltmo = \bigg\{ \chi : \R^d \to \R \mbox{ measurable },\;\intv{(\chi (v))^2 \mo(v)} < \infty \bigg\} \, ,  
\]
and
\[
\homo = \bigg\{ \chi : \R^d \to \R \mbox{ measurable },\;\intv{[\;(\chi (v))^2 + |\nabla _v \chi |^2\;]\mo(v)} < \infty \bigg\} \, .
\]
The nonlinear operator $Q$ should be understood in the distributional sense, and is defined for any particle density $f = f(v)$ in the domain
\begin{align*}
D(Q) & = \Big\{ f : \R^d \to \R_+ \mbox{ measurable }, \; f/\mof \in H^1_{\mof} \Big\} \\
& = \left \{ f : \R^d \to \R_+ \mbox{ measurable },\; \intv{\left \{\left ( \frac{f}{\mof } \right ) ^2  + \left |\nabla _v  \left ( \frac{f}{\mof} \right )  \right |^2 \right \}\mof (v)} < \infty\right \}.
\end{align*}
We introduce the usual scalar products 
\[
(\chi, \theta)_{\mo} = \intv{\chi (v) \theta (v) \mo (v)},\;\;\chi, \theta \in \ltmo \, ,
\]
\[
((\chi, \theta))_{\mo} = \intv{( \chi (v) \theta (v) + \nabla _v \chi \cdot \nabla _v \theta) \mo (v)},\;\;\chi, \theta \in \homo \, ,
\]
and we denote by $|\cdot |_{\mo}, \|\cdot \|_{\mo}$ the associated norms. We make the following hypotheses on the potential $V$. We assume that for any $\Omega \in \sphere$ we have
\begin{equation}
\label{eq:Z} Z_\Omega = \intv{\exp\left (- \frac{1}{\sigma} \left [ \frac{|v - \Omega|^2}{2} + V(|v|)\right ]\right )} < \infty \, .
\end{equation}
Clearly \eqref{eq:Z} holds true for the potentials $V_{\alpha, \beta}$. Notice that in that case $1 \in \ltmo$ and $|1|_{\mo} = 1$ for any $\Omega \in \sphere$. Moreover, we need a Poincar\'e inequality, that is, for any $\Omega \in \sphere$ there is $\lambda _\Omega >0$ such that for all $\chi \in \homo$ we have
\begin{equation}
\label{eq:Poincare}
\sigma \intv{|\nabla _v \chi |^2 \mo (v)} \geq \lambda _\Omega \intv{\left |\chi (v) - \intvp{\chi (\vp) \mo (\vp)}\right |^2 \mo (v)} \, .
\end{equation}
A sufficient condition for \eqref{eq:Poincare} to hold comes from the well-known equivalence between the 
Fokker-Planck and Schr\"odinger operators (see for instance \cite{BonCarGouPav16}). Namely, for any $\Omega \in \sphere$ we have
\[
- \frac{\sigma}{\sqrt{\mo}} \Divv\left ( \mo \nabla _v \left ( \frac{u}{\sqrt{\mo}}\right ) \right ) = - \sigma \Delta _v u + \left [ \frac{1}{4\sigma} |\nabla _v \Phio |^2 - \frac{1}{2} \Delta _v \Phio \right ] u \, .
\]
The operator $\mathcal{H}_\Omega = - \sigma \Delta _v + \left [ \frac{1}{4\sigma} |\nabla _v \Phio |^2 - \frac{1}{2} \Delta _v \Phio \right ]$ is defined in the domain
\[
D(\mathcal{H}_\Omega) = \left \{u \in L^2 (\R^d),\;\left [ \frac{1}{4\sigma} |\nabla _v \Phio |^2 - \frac{1}{2} \Delta _v \Phio \right ]u \in L^2 (\R^d)  \right \} \, .
\]
Using classical results for Schr\"odinger operators (see for instance Theorem XIII.67 in \cite{ReedSimon78}), we have a spectral 
decomposition of the operator $\mathcal{H}_\Omega$ under suitable confining assumptions.
\begin{lemma}
\label{RSVol4}
Assume that for $\Phio$ defined in \eqref{eq:pot} the function $v \to \frac{1}{4\sigma} |\nabla _v \Phio |^2 - \frac{1}{2}\Delta _v \Phio$ satisfies the following:
\begin{itemize}
 \item[a)] it belongs to $L^1_{\mathrm{loc}}(\R^d)$,
 \item[b)] it is bounded from below,
 \item[c)] \[
\lim _{|v| \to \infty} \left [ \frac{1}{4\sigma} |\nabla _v \Phio |^2 - \frac{1}{2}\Delta _v \Phio \right ]=  \infty \, .
\]
\end{itemize}
Then $\mathcal{H}_\Omega ^{-1}$ is a self adjoint compact operator in $L^2(\R^d)$ and $\mathcal{H}_\Omega$ admits a spectral 
decomposition, that is a nondecreasing sequence of real numbers $(\lambda _\Omega ^n)_{n\in \N}$, $\lim _{n \to \infty} \lambda _\Omega ^n = \infty$, 
and a $L^2(\R^d)$-orthonormal basis $(\psi _\Omega ^n)_{n \in \N}$ such that $\mathcal{H}_\Omega \psi _\Omega ^n = \lambda _\Omega ^n \psi _\Omega ^n, n \in \N$, $\lambda _\Omega ^0 = 0$, $\lambda _\Omega ^1 >0$. 
\end{lemma}
Let us note that the spectral gap of the Schr\"odinger operator $\mathcal{H}_\Omega$ is the Poincar\'e 
constant in the Poincar\'e inequality \eqref{eq:Poincare}. Notice also that the hypotheses in 
Lemma \ref{RSVol4} are satisfied by the potentials $V_{\alpha, \beta}$, 
and therefore \eqref{eq:Poincare} holds true in that case. It is easily seen that the set of 
equilibrium distributions of $Q$ is parametrized by $d$ parameters as stated in the following result. 
\begin{lemma}
\label{Equilibrium}
Let $f = f(v)\geq 0$ be a function in $D(Q)$. Then $f$ is an equilibrium for $Q$ if and only if there are 
$(\rho, \Omega) \in \R_+ \times \left ( \sphere \cup \{0\} \right )$ such that $f = \rho \mo$. Moreover 
we have $\rho = \rho[f]: = \intv{f(v)}$ and $\Omega = \Omega[f]$. 
\end{lemma}

\begin{proof}\\
If $f$ is an equilibrium for $Q$, we have
\[
\sigma \intv{\left |\nabla _v \left ( \frac{f}{\mof} \right )\right | ^2 \mof (v)} = 0 \, ,
\]
and therefore there is $\rho \in \R$ such that $f = \rho \mof$. Obviously $\rho = \intv{f(v)}\geq 0$ and $\Omega [f] \in \sphere \cup \{0\}$. Conversely, we claim that for any $(\rho, \Omega)\in \R_+ \times ( \sphere \cup \{0\})$, the particle density $f = \rho \mo$ is an equilibrium for $Q$. Indeed, we have
\[
\sigma \nabla _v (\rho \mo ) + \rho \mo (v - \Omega + \nabla _v V ) = \rho ( \sigma \nabla _v \mo + \mo \nabla _v \Phio ) = 0 \, .
\]
We are done if we prove that $\Omega [f] = \Omega$. If $\Omega = 0$, it is easily seen that 
\[
\intv{f(v)v} = 
\rho\intv{\frac{1}{Z_0} \exp \left ( - \frac{\Phi _0(v)}{\sigma}\right ) v} = \frac{\rho}{Z_0} \intv{\exp\left (- \frac{1}{\sigma}\left ( \frac{|v|^2}{2} + V(|v|)\right)   \right )v} = 0 \, ,
\]
implying $\Omega[f] = 0 = \Omega$. Assume now that $\Omega \in \sphere$. For any $\xi \in \sphere$, $\xi \cdot \Omega = 0$, we consider the orthogonal transformation $\O _\xi = I_d - 2\xi \otimes \xi$. Thanks to the change of variable $v = \O _\xi \vp$, we write
\begin{align*}
\intv{(v\cdot \xi)f(v)} & = \frac{\rho}{\zo}\intv{(v \cdot \xi) \mo (v)} = \frac{\rho}{\zo}\intvp{(\O _\xi \vp \cdot \xi) \mo (\O _\xi \vp)}\\
& = - \frac{\rho}{\zo}\intvp{(\vp \cdot \xi) \mo (\vp) } = - \intvp{(\vp \cdot \xi) f(\vp) } \, ,
\end{align*}
where we have used the radial symmetry of $V$, $\O _\xi \xi = - \xi$ and $\O _\xi \Omega = \Omega$. We deduce that $\intv{f(v)v} = \intv{(v\cdot \Omega) f(v)}\;\Omega$. We claim that $\intv{(v\cdot \Omega) f(v)}>0$. Indeed we have
\begin{align*}
\intv{(v\cdot \Omega) f(v)} & = \frac{\rho}{\zo} \int _{v \cdot \Omega >0} (v\cdot \Omega) \mo (v)\;\mathrm{d}v + \frac{\rho}{\zo} \int _{v \cdot \Omega <0} (v\cdot \Omega) \mo (v)\;\mathrm{d}v\\
& = \frac{\rho}{\zo} \int _{v \cdot \Omega >0}(v\cdot \Omega)\left [\exp\left (- \frac{\Phio(v)}{\sigma} \right ) - \exp \left (- \frac{\Phio(-v)}{\sigma} \right )  \right ] \;\mathrm{d}v \, .
\end{align*}
Obviously, we have for any $v \in \R^d$ such that $v \cdot \Omega>0$
\[
- \frac{\Phio (v)}{\sigma} + \frac{\Phio (-v)}{\sigma}= - \frac{|v - \Omega|^2}{2\sigma} + \frac{|-v - \Omega|^2}{2\sigma} = 2 \frac{v \cdot \Omega}{\sigma} > 0 \, ,
\]
implying that $\intv{(v\cdot \Omega) f(v)} >0$ and
\[
\Omega[f] = \frac{\intv{f(v)v}}{\left | \intv{f(v)v} \right |}= \frac{\intv{\;(v \cdot \Omega) f(v)}\;\Omega}{\left | \intv{\;(v \cdot \Omega)f(v)}\;\Omega \right |} = \Omega \, .
\]
\end{proof}


\section{Characterization of the collision invariants}\label{CharCollInv}
In \cite{DM08}, the following notion of generalized collision invariant (GCI) has been introduced.
\begin{defin} (GCI)\\
\label{GeneralizedCollisionInvariant}
Let $\Omega \in \sphere$ be a fixed orientation. A function $\psi = \psi (v)$ is called a generalized collision 
invariant of $Q$ associated to $\Omega$, if and only if
\[
\intv{Q(f)(v) \psi (v) } = 0 \, ,
\]
for all $f$ such that $(I_d - \Omega \otimes \Omega) \intv{f(v)v} = 0$, that is such that $\intv{f(v)v } \in \R \Omega$. 
\end{defin}

\

In order to obtain the hydrodynamic limit of \eqref{eq:main}, for any fixed $(t,x) \in \R_+ \times \R^d$, we multiply \eqref{eq:1exp} 
by a function $v \to \psi _{t,x} (v)$ and integrate with respect to $v$ yielding
\begin{align}
\label{eq:20}
\intv{\partial _t \fz (t,x,v) \psi _{t,x} (v)} + \intv{v \cdot \nabla _x \fz (t,x,v) \psi _{t,x} (v)} & = \intv{\calL _{f(t,x,\cdot)}(\fo (t,x,\cdot)) \psi _{t,x} (v)}\\
& = \intv{\fo (t,x,v) (\calL _{f(t,x,\cdot)}^\star \psi _{t,x})(v) } \, .\nonumber 
\end{align}
The above computation leads naturally to the following extension of the notion of collision invariant, see also \cite{BosCar15}.
\begin{defin}
\label{CollisionInvariant}
Let $f = f(v)\geq 0$ be an equilibrium of $Q$. A function $\psi = \psi (v)$ is called a collision invariant for $Q$ 
associated to the equilibrium $f$, if and only if $\calL _f ^\star \psi = 0$, that is
\[
\intv{(\calL _f g )(v) \psi (v) } = 0\;\;\mbox{ for any function } g = g(v) \, .
\]
\end{defin}

\

We are looking for a good characterization of the linearized collision operator $\calL _f$ and its adjoint with respect to the leading order particle density $\fz$. Motivated by \eqref{eq:0exp}, 
we need to determine the structure of the equilibria of $Q$ which are given by Lemma \ref{Equilibrium}.\\
By Lemma \ref{Equilibrium}, we know that for any $(t,x) \in \R_+ \times \R^d$, there are $(\rho (t,x), \Omega (t,x)) \in \R_+ \times (\sphere  \cup \{0\})$ such that $\fz(t,x,\cdot) = \rho (t,x) M_{\Omega (t,x)}$, where
\[
\rho (t,x) = \rho [\fz (t,x,\cdot)] \, \, \text{ and } \, \, \Omega (t,x) = \Omega[\fz(t,x,\cdot)] \, .
\]
The evolution of the macroscopic quantities $\rho$ and $\Omega$ follows from \eqref{eq:1exp} and \eqref{eq:20}, by appealing to the moment method \cite{BarGolLevI93,BarGolLevII93,BosDCDS15,BosIHP15,Lev93,Lev96}. 
Next, we explicitly determine the linearization of the collision operator $Q$ around its equilibrium distributions. For any orientation $\Omega \in \sphere \cup \{0\}$ we introduce the pressure tensor 
$$
\calmo := \intv{(v-\Omega) \otimes (I_d - \Omega \otimes \Omega) (v - \Omega) \mo (v) }\, ,
$$ 
and the quantity
$$
c_1 := \intv{\;(v \cdot \Omega) \mo (v)}>0\,.
$$
We will check later, see Lemma \ref{Spectral}, that the pressure tensor $\calmo$ is symmetric.
\begin{proposition}
\label{Linear}
Let $f = f(v) \geq 0$ be an equlibrium distribution of $Q$ with nonvanishing orientation, that is
\[
f = \rho \mo,\;\; \text{ where } \; \; \rho = \rho[f],\;\; \text{ and } \; \; \Omega = \Omega[f] \in \sphere \, .
\]
\begin{enumerate}
\item
The linearization $\calL _f = \d Q_f$ is given by
\[
\calL _f g = \Divv \left \{ \sigma \nabla _v g + g \nabla _v \Phio - \frac{f}{\intv{\;(v\cdot \Omega) f (v)}} P_f \intv{g(v) v } \right \} \, ,
\]
where $\pf := I_d - \Omega[f] \otimes \Omega [f]$ is the orthogonal projection onto $\{\xi \in \R^d : \;\xi \cdot \Omega[f] = 0\}$. In particular $\calL_{\rho \mo} = \calL _{\mo}$. 
\item
The formal adjoint of $\calL _f$ is given by
\begin{equation}
\label{eq:adjoint}
\calL _f ^\star \psi = \sigma \frac{\Divv (\mo \nabla _v \psi)}{\mo} + P_f v \cdot W[\psi],\;\;W[\psi] := \frac{\intv{\mo (v)\nabla _v \psi}}{\intv{\;(v\cdot \Omega)\mo (v)}} \, .
\end{equation}
\item
We have the identity
\[
\calL _f (f(v-\Omega)) = \sigma \nabla _v f - \Divv 
\left ( f \frac{\calmo}{c_1}\right ).
\]
Note that $\Divv$ refers to the divergence operator acting on matrices defined as applying the divergence operator over rows.
\end{enumerate}
\end{proposition}
\begin{proof}$\;$\\
1. By standard computations we have
\[
\calL _f g  = \frac{\d}{\d s} \Big|_{s= 0} Q(f+sg) = \Divv \left \{\sigma \nabla _v g + g (v - \Omega[f] + \nabla _v V) - f \frac{\d}{\d s}\Big|_{s= 0} \Omega[f+sg]  \right \} \, ,
\]
and
\[
\frac{\d}{\d s}\Big|_{s= 0}\Omega [f+sg] = \frac{(I_d - \Omega[f] \otimes \Omega[f])}{\left | \intv{f(v) v }\right | } \intv{ g(v)v} \, .
\]
Therefore we obtain
\[
\calL _f g = \Divv \left \{ \sigma \nabla _v g + g \nabla _v \Phio - \frac{f}{\intv{\;(v\cdot \Omega) f (v)}} P_f \intv{g(v) v } \right \} \, . 
\]
2. We have
\begin{align*}
&\intv{(\calL _f g) (v)\psi(v)  } \\
& \qquad = -  \intv{\left \{ \sigma \nabla _v g + g \nabla _v \Phio - \frac{f}{\intvp{\;(\vp\cdot \Omega) f (\vp)}} P_f \intvp{g(\vp) \vp } \right \}\cdot \nabla _v \psi }\\
& \qquad = \intv{g[ \sigma \Divv \nabla _v \psi - \nabla _v \psi \cdot \nabla _v \Phio]} + \intvp{g(\vp) \pf \vp \cdot \frac{\intv{f(v) \nabla _v \psi(v)}}{\intv{\;(v\cdot \Omega) f(v)}}} \, ,
\end{align*}
implying 
\[
\calL _f ^\star \psi = \sigma \frac{\Divv (\mo \nabla _v \psi)}{\mo} + P_f v \cdot W[\psi] \, .
\]
3. For any $i \in \{1,...,d\}$ we have
\begin{align*}
\calL _f (f(v-\Omega)_i)& = \Divv \left \{(v - \Omega)_i (\sigma \nabla _v f + f \nabla _v \Phio) + \sigma f e_i - f\frac{\intvp{\mo(\vp) (\vp - \Omega)_i \pf \vp}}{\intvp{\;(\vp \cdot \Omega)\mo (\vp)}}   \right \} \, ,
\end{align*}
and therefore, since $f = \rho M_\Omega$ satisfies $\sigma \nabla _v f + f \nabla _v \Phio = 0$, we get
\begin{align*}
\calL _f (f(v-\Omega)) & = \sigma \nabla _v f - \Divv \left ( f \frac{\intvp{\mo (\vp) (\vp - \Omega) \otimes P_f \vp}}{\intvp{\;(\vp \cdot \Omega) \mo (\vp)}} \right ) \\
& = \sigma \nabla _v f - \Divv\left ( f \frac{\calmo}{c_1}\right ).
\end{align*}
\end{proof}

\
Notice that at any $(t,x) \in \R_+ \times \R^d$, the function $h = 1$ is a collision invariant for $Q$, associated to $f(t,x,\cdot)$. Indeed, for any $g = g(v)$ we have
\[
\intv{Q(f(t,x,\cdot) + sg)} = 0 \, ,
\]
implying that $\intv{\;(\calL _{f(t,x,\cdot)}g)(v)} = 0$ and therefore 
$\calL _{f(t,x, \cdot)} ^\star 1 = 0$, for all $(t,x) \in \R_+ \times \R^d$. Once we have determined a collision 
invariant $\psi = \psi (t,x,v)$ at any $(t,x) \in \R_+ \times \R^d$, we deduce, thanks to \eqref{eq:20}, 
a balance for the macroscopic quantities $\rho (t,x) = \rho[f(t,x,\cdot)]$ and $\Omega (t,x) = \Omega[f(t,x,\cdot)]$,
given by the relationship
\begin{equation}
\label{eq:balance} 
\intv{\partial _t (\rho M_{\Omega(t,x)}) \psi (t,x,v) } + \intv{v \cdot \nabla _x (\rho M_{\Omega(t,x)}) \psi (t,x,v) } = 0 \, .
\end{equation}
When taking as collision invariant the function $h(t,x,v) = 1$, we obtain the local mass conservation equation
\begin{equation}
\label{eq:ContEqu}
\partial _t \rho + \Divx \left ( \rho \intv{\;(v \cdot \Omega (t,x)) M_{\Omega (t,x)}(v)} \;\Omega \right ) = 0 \, .
\end{equation}
As usual, we are looking also for the conservation of the total momentum, however, the 
nonlinear operator $Q$ does not preserve 
momentum. In other words, $v$ is not a collision invariant. Indeed, if 
$f = \rho \mo$ is an equilibrium with nonvanishing orientation, we have
\[
\calL _f ^\star v = \sigma \frac{\nabla _v \mo}{\mo} + \frac{\pf v}{\intvp{\;(\vp \cdot \Omega) \mo (\vp)}} = - \nabla _v \Phio +  \frac{\pf v}{\intvp{\;(\vp \cdot \Omega) \mo (\vp)}} \, ,
\]
and therefore $v$ is not a collision invariant. 

We concentrate next on the resolution of \eqref{eq:adjoint}. We will use the notation $\partial_v \xi =\left(\tfrac{\partial \xi_i}{\partial v_j}\right)$ for the Jacobian matrix of a vector field $\xi$ and
$ \Divv$ for the divergence operator in $v$ of both vectors and matrices with the convention of taking the divergence over the rows of the matrix. With this convention, we have
\begin{equation}
\label{EquIntByParts}
\intv{g \, \Divv A} = - \intv{A \, \nabla_v g} \qquad \mbox{and} \qquad
\intv{\xi \, \Divv \eta} = - \intv{\partial_v \xi \, \eta}
\end{equation}
for all smooth functions $g$, vector fields $\xi,\eta$, and matrices $A$. 
We now focus in finding a parameterization of the kernel of the operator $\calL _f ^\star$.
\begin{lemma}
\label{AdjointRes}
Let $f = \rho \mo$ be an equilibrium of $Q$ with nonvanishing orientation. 
The following two statements are equivalent:
\begin{enumerate}
\item
$\psi = \psi (v)$ is a collision invariant for $Q$ associated to $f$.
\item
$\psi$ satisfies
\begin{equation}
\label{eq:22}\sigma \frac{\Divv (\mo \nabla _v \psi)}{\mo} + \pf v \cdot W = 0 \, ,
\end{equation}
for some vector $W \in \ker (\calmo - \sigma c_1 I_d)$. 
\end{enumerate}
Moreover, the linear map $W : \ker (\calL _f ^\star) \to \ker (\calmo - \sigma c_1 I_d)$, with $W[\psi] := \intv{\mo (v) \nabla _v \psi} / c_1$ induces an isomorphism between the vector spaces $\ker (\calL _f ^\star)/\ker W$ and $\ker (\calmo - \sigma c_1 I_d)$, where $\ker W$ is the set of the constant functions.
\end{lemma}
\begin{proof} $\;$\\
1.$\implies$2. Since $\psi$ is a collision invariant associated to $f$, i.e. $\calL _f ^\star \psi = 0$,  and by the third statement in Proposition \ref{Linear} we deduce (using also the first formula in \eqref{EquIntByParts} with $f \calmo /c_1$ and $\psi$)
\begin{align*}
0 & = \intv{\calL _f ^\star \psi  \; f (v - \Omega)} = \intv{\psi (v) \calL _f (f(v-\Omega))} \\
& = \intv{\psi (v) \left [\sigma \nabla _v f - \Divv \left ( f \frac{\calmo}{c_1}\right )   \right ] } = - \sigma \intv{f(v) \nabla _v \psi } + \calmo \intv{\frac{f(v) \nabla _v \psi}{c_1}} \\
& = - \rho \sigma c_1 W[\psi] + \rho \calmo W[\psi] \, .
\end{align*}
Note that if $\rho = 0$ then $f = 0$ and $ \int v f \d v = 0$, implying that $\Omega [ f] = 0$. Hence, 
since $\Omega \neq 0$, we have $\rho > 0$ and thus $W[\psi] \in \ker (\calmo - \sigma c_1 I_d)$, saying that \eqref{eq:22} holds true with $W = W[\psi] \in \ker(\calmo - \sigma c_1 I_d)$.

\

\noindent
2.$\implies$1. Let $\psi$ be a function satisfying \eqref{eq:22} for some vector $W \in \ker (\calmo - \sigma c_1 I_d)$. Multiplying \eqref{eq:22} by $f(v-\Omega)$ and integrating with respect to $v$ yields (thanks to the second formula in \eqref{EquIntByParts}) 
\[
- \sigma \rho\intv{\partial _v (v - \Omega) \nabla _v \psi  \mo (v) } + \rho \calmo W = 0 \, ,
\]
which implies $W[ \psi] = W$ since $\calmo W = \sigma c_1 W$ by the assumption $W \in \ker (\calmo - \sigma c_1 I_d)$. Therefore $\psi$ is a collision invariant for $Q$, associated to $f$
\[
\calL _f ^\star \psi = \sigma \frac{\Divv (\mo \nabla _v \psi)}{\mo} + P_f v \cdot W[\psi] = \sigma \frac{\Divv (\mo \nabla _v \psi)}{\mo} + P_f v \cdot W = 0 \, .
\]
\end{proof}
\begin{remark}
\label{c1}
For any non negative measurable function $\chi = \chi (c,r) : ]-1, +1[ \times ] 0, \infty[ \to \R$ and any $\Omega \in \sphere$, for $d \geq 2$, we have
\begin{equation*}
\label{eq:23}
\intv{\chi \left ( \frac{v \cdot \Omega}{|v|},|v| \right )} = |\mathbb{S}^{d-2}| \int _{\R_+}\intth{\chi(\cos \theta, r) r ^{d-1} \sin ^{d-2} \theta }\d r \, ,
\end{equation*}
where $|\mathbb{S}^{d-2}|$ is the surface of the unit sphere in $\R^{d-1}$, for $d \geq 3$, and $|\mathbb{S}^0| = 2$
for $d = 2$. In particular we have the formula 
\begin{align}
\label{eq:34}
\intv{\chi \left ( \frac{v \cdot \Omega}{|v|},|v| \right )\mo (v)}& = \frac{\int _{\R_+}r^{d-1} \intth{\chi(\cos \theta, r) e(\cos \theta, r) \sin ^{d-2} \theta }\d r}{\int _{\R_+}r^{d-1} \intth{e(\cos \theta, r) \sin ^{d-2} \theta }\d r}\\
& = \frac{\int_{\R_+}r^{d-1} \int _{-1} ^{+1} \chi (c,r) e(c,r) (1 - c^2) ^{\frac{d-3}{2}}\d c \d r}{\int_{\R_+}r^{d-1} \int _{-1} ^{+1}  e(c,r) (1 - c^2) ^{\frac{d-3}{2}}\d c \d r} \, , \nonumber 
\end{align}
where $e(c,r) = \exp(rc/\sigma) \exp ( - (r^2 + 1)/(2\sigma) - V(r)/\sigma )$. 
\end{remark}
Notice that thanks to \eqref{eq:34} the coefficient $c_1$ does not depend upon $\Omega \in \sphere$
\begin{align*}
c_1 & = \frac{\intv{\;(v \cdot \Omega) \exp \left ( - \frac{|v-\Omega|^2}{2\sigma} - \frac{V(|v|)}{\sigma} \right )}}{\intv{\exp \left ( - \frac{|v-\Omega|^2}{2\sigma} - \frac{V(|v|)}{\sigma} \right )}} = \frac{\int_{\R_+}r^d\intth{\cos \theta \;e(\cos \theta, r)\sin ^{d-2} \theta}\d r}{\int_{\R_+}r^{d-1}\intth{e(\cos \theta, r)\sin ^{d-2} \theta}\d r} \, .
\end{align*}
In order to determine all the collision invariants, we focus on the spectral decomposition of the pressure tensor $\calmo$ for any $\Omega \in \sphere$. In particular, the next lemma will imply the symmetry of the pressure tensor.
\begin{lemma} (Spectral decomposition of $\calmo$)
\label{Spectral}
For any $\Omega \in \sphere$ we have $\calmo = \sigma c_1 (I_d - \Omega \otimes \Omega)$. In particular we have $\ker (\calmo - \sigma c_1 I_d) = (\R \Omega )^\perp$ and thus $\mathrm{dim} \left ( \ker (\calL _f ^\star ) /\ker W \right ) = \mathrm{dim} \ker ( \calmo - \sigma c_1 I_d) = d-1$, cf. Lemma \ref{AdjointRes}.
\end{lemma}
\begin{proof}
Let us consider $\{E_1, \ldots, E_{d-1}\}$ an orthonormal basis of $(\R\Omega) ^\perp$. By using the decomposition 
\[
v - \Omega = ( \Omega \otimes \Omega ) (v-\Omega) + \sumi ( E_i \otimes E_i ) (v-\Omega) = ( \Omega \otimes \Omega ) (v-\Omega) + \sumi ( E_i \otimes E_i ) v \, ,
\]
one gets
\begin{equation}
\label{eq:26}
\calmo = \intv{\left [ ( \Omega \otimes \Omega ) (v-\Omega) + \sumi ( E_i \otimes E_i ) v
\right ]\otimes \left [ \sumj ( E_j \otimes E_j ) v \right ] \mo (v)} \, .
\end{equation}
We claim that the following equalities hold true
\begin{equation}
\label{eq:24}
\intv{[\Omega \cdot (v - \Omega)] (E_j \cdot v) \mo (v) } = 0 \, ,\;\;1\leq j \leq d-1 \, ,
\end{equation}
\begin{equation}
\label{eq:25}
\intv{(E_i \cdot v) (E_j \cdot v) \mo (v) } = \delta _{ij} \intv{\frac{|v|^2 - (v \cdot \Omega)^2}{d-1} \mo (v)},\;\;1\leq i, j \leq d-1 \, .
\end{equation}
Formula \eqref{eq:24} is obtained by using the change of variable $v = (I_d - 2 E_j \otimes E_j) \vp$. It is easily seen that 
\[
\Omega \cdot ( v - \Omega) = \Omega \cdot ( \vp - \Omega),\;\;E_j \cdot v = - E_j \cdot \vp,\;\;\mo (v) = \mo (\vp),\;\;1\leq j \leq d-1 \, ,
\]
and therefore we have
\[
\intv{[\Omega \cdot (v - \Omega)] (E_j \cdot v) \mo (v) }= - \intvp{[\Omega \cdot (\vp - \Omega)] (E_j \cdot \vp) \mo (\vp) }
\]
which implies \eqref{eq:24}. For the formulae \eqref{eq:25} with $i \neq j$, we appeal to the orthogonal transformation
\[
v = \O _{ij} \vp,\;\;\O _{ij} = \Omega \otimes \Omega + \sum _{k \notin \{i,j\}} E_k \otimes E_k + E_i \otimes E_j - E_j \otimes E_i \, .
\]
Notice that $\O _{ij} \xi = \xi$, for all $\xi \in ( \mathrm{span}\{E_i, E_j\} ) ^\perp$, $\O _{ij} E_i = - E_j$, $\O _{ij} E_j = E_i$ and therefore 
\[
(E_i \cdot v) (E_j \cdot v) = - (E_j \cdot \vp) (E_i \cdot \vp) \, .
\]
After this change of variable we deduce that
\[
\intv{(E_i \cdot v) (E_j \cdot v ) \mo (v) } = 0,\;\;1\leq i, j \leq d-1,\;\;i \neq j \, ,
\]
and also 
\[
\intv{(E_i \cdot v)^2 \mo (v) } = \intv{(E_j \cdot v)^2 \mo (v) },\;\;1\leq i, j \leq d-1 \, .
\]
Thanks to the equality $\sumi (E_i \cdot v)^2 = |v|^2 - (v \cdot \Omega )^2$, one gets
\[
\intv{(E_i \cdot v) (E_j \cdot v) \mo (v) } = \delta _{ij} \intv{\frac{|v|^2 - (v \cdot \Omega)^2}{d-1} \mo (v)},\;\;1\leq i, j \leq d-1 \, .
\]
Coming back to \eqref{eq:26} we obtain
\begin{align*}
\calmo = \sumi \left ( \intv{(E_i \cdot v)^2 \mo (v)} \right ) E_i \otimes E_i = \intv{\frac{|v|^2 - (v \cdot \Omega) ^2}{d-1} \mo (v)} (I_d - \Omega \otimes \Omega) \, .
\end{align*}
We are done if we prove that
\begin{equation*}
\label{eq:27}
\intv{\frac{|v|^2 - (v \cdot \Omega) ^2}{d-1} \mo (v)} = \sigma c_1 \, .
\end{equation*}
Notice that, using \eqref{eq:grad_m}: 
\[
\big( ( |v|^2 I_d - v \otimes v ) \Omega \big) \cdot \nabla _v \mo = \frac{|v|^2 - (v \cdot \Omega)^2}{\sigma} \mo (v) \, ,
\]
and therefore
\begin{align*}
\intv{\frac{|v|^2 - (v \cdot \Omega)^2}{\sigma} \mo (v)} & = - \intv{\Divv [ (|v|^2 I_d - v \otimes v) \Omega ]\mo (v) } \\
& = - \intv{[2(v \cdot \Omega) - d(v \cdot \Omega) - (v \cdot \Omega)] \mo (v) } \\
& = (d-1) c_1 \, .
\end{align*}
\end{proof}
By Lemmas \ref{AdjointRes} and \ref{Spectral} the computation of the collision invariants is reduced to the 
resolution of \eqref{eq:22} for any $W \in (\R \Omega) ^\perp$. Hence, if we denote by $E_1, E_2, \ldots , E_{d-1}$
any orthonormal basis of $(\R \Omega )^{\perp}$, we obtain a set of $d-1$ collision invariants $\psi_{E_1}, \psi_{E_2}, \ldots , \psi_{E_{d-1}}$ for $Q$ associated to the equilibrium distribution $f$ such that $E_i=W[\psi_{E_i}]$, $i=1,\dots,d-1$. 
This set of collision invariants forms a basis for the $\ker(\calL _f ^\star)$.
In the next section we will characterize this set of collision invariants and provide and easy manner to compute them (see Lemma \ref{InvField}).

We conclude this section by showing that in our case the set of all GCIs of the operator $Q$ 
coincide with the kernel of the 
operator $\calL _f ^\star$.
\begin{theorem}
\label{EquiCollInv}
Let $\mo$ be an equilibrium of $Q$ with nonvanishing orientation $\Omega \in \sphere$. The set of collision invariants of $Q$ 
associated to $\mo$ coincides with the set of the generalized collision invariants of $Q$ associated to $\Omega$. 
\end{theorem}
\begin{proof}\\
Let $\psi = \psi (v)$ be a generalized collision invariant of $Q$ associated to $\Omega$. We denote by 
$\{e_1, \ldots, e_d\}$ the canonical basis of $\R^d$. For any $f = f(v)$ satisfying 
$(I_d - \Omega \otimes \Omega) e_i \cdot \intv{f(v) v } = 0, 1\leq i \leq d$, we have
\[
\intv{f(\sigma \Delta _v \psi - \nabla _v \Phio \cdot \nabla _v \psi )} = \intv{Q(f)(v) \psi (v)} = 0 \, .
\]
Therefore the linear form $f \to \intv{f(\sigma \Delta _v \psi - \nabla _v \Phio \cdot \nabla _v \psi)}$ is a 
linear combination of the linear forms $f \to (I_d - \Omega \otimes \Omega) e_i \cdot \intv{f(v) v }$. We 
deduce that there is a vector $\tilde{W} = (\tilde{W}_1,\ldots, \tilde{W}_d) \in \R^d$ such that 
\[
\intv{f(\sigma \Delta _v \psi - \nabla _v \Phio \cdot \nabla _v \psi )} + (I_d - \Omega \otimes \Omega) \tilde{W} \cdot \intv{f(v) v }= 0 \, ,
\]
for any $f$ and thus
\[
\sigma \Delta _v \psi - \nabla _v \Phio \cdot \nabla _v \psi  + (I_d - \Omega \otimes \Omega)v\cdot  \tilde{W}= 0 \, ,
\]
implying that $\psi$ satisfies \eqref{eq:22} with the vector $W =(I_d - \Omega \otimes \Omega)\tilde{W} \in (\R \Omega) ^\perp$, 
that is, $\psi$ is a collision invariant of $Q$ associated to $\mo$. 

Conversely, let $\psi = \psi (v)$ be a collision invariant of $Q$ associated to $\mo$. By Lemmas 
\ref{AdjointRes}, \ref{Spectral} we know that there is $W \in (\R \Omega)^\perp$ such that 
\[
\sigma \Delta _v \psi - \nabla _v \Phio \cdot \nabla _v \psi + v \cdot W = 0 \, .
\]
Multiplying by any function $f$ such that $(I_d - \Omega \otimes \Omega) \intv{f(v)v} = 0$ one gets
\[
\intv{Q(f) (v) \psi (v)} = \intv{f(v) ( \sigma \Delta _v \psi - \nabla _v \Phio \cdot \nabla _v \psi)} =
- \intv{f(v)v} \cdot W = 0 \, ,
\]
implying that $\psi$ is a generalized collision invariant of $Q$ associated to $\Omega$. 
\end{proof}

\section{Identification of the collision invariants}\label{IdentCollInv}

In this section we investigate the structure of the collision invariants 
of $Q$ associated to an equilibrium distribution $f = \rho \mo$. By Lemmas \ref{AdjointRes}, 
\ref{Spectral}, we need to solve the elliptic problem
\begin{equation}
\label{eq:28}
- \sigma \Divv (\mo \nabla _v \psi) = (v \cdot W) \mo (v),\;\;v \in \R^d \, ,
\end{equation}
for any $W \in (\R\Omega )^\perp$. We appeal to a variational formulation by considering the continuous bilinear 
symmetric form $a_\Omega: \homo \times \homo \to \R$ defined as
\[
a_\Omega (\chi, \theta) = \sigma \intv{\nabla _v \chi \cdot \nabla _v \theta\; \mo (v)},\;\;\chi, \theta \in \homo \, ,
\]
and the linear form $l : \homo \to \R$, $l (\theta) = \intv{\theta(v) (v \cdot W) \mo (v)}, \theta \in \homo$. Notice that $l$ is well defined and bounded on $\homo$ provided that the additional hypothesis $|v| \in \ltmo$ holds true, that is
\begin{equation*}
\label{eq:K}
\intv{|v|^2 \exp \left ( - \frac{|v- \Omega|^2}{2\sigma} - \frac{V(|v|)}{\sigma} \right )} < \infty \, .
\end{equation*}
The above hypothesis is obviously satisfied by the potentials $V_{\alpha, \beta}$. We say that $\psi \in \homo$ is a variational solution of \eqref{eq:28} if and only if 
\begin{equation}
\label{eq:29}
a_\Omega (\psi, \theta) = l(\theta)\;\;\mbox{ for any } \theta \in \homo \, .
\end{equation}

\begin{proposition}
A necessary and sufficient condition for the existence and uniqueness of variational solution to \eqref{eq:28} is
\begin{equation}
\label{eq:30}
\intv{(v \cdot W) \mo (v)} = 0 \, .
\end{equation}
\end{proposition}

\begin{proof}
The necessary condition for the solvability of \eqref{eq:28} is obtained by taking $\theta = 1$ (which belongs to $\homo$ thanks to \eqref{eq:Z}) in \eqref{eq:29} leading to \eqref{eq:30}. This condition is satisfied for any $W \in (\R\Omega)^\perp$ since we have
\[
\intv{(v \cdot W)\mo (v)} = \intv{(v \cdot \Omega)\mo (v)}\;(\Omega \cdot W) = 0 \, .
\]
The condition \eqref{eq:30} also guarantees the solvability of \eqref{eq:28}. Indeed, under the hypotheses \eqref{eq:Z}, \eqref{eq:Poincare}, the bilinear form $a_\Omega$ is coercive on the Hilbert space $\thomo : = \{\chi \in \homo\;:\; ((\chi,1))_{\mo} = 0\}$, i.e. we have:
\[
a_\Omega (\chi, \chi) \geq \frac{\min\{\sigma,\lambda _\Omega\}}{2} \|\chi \|^2 _{\mo},\;\;\chi \in \thomo \, .
\]
Applying the Lax-Milgram lemma to \eqref{eq:29} with $\psi$, $\theta \in \thomo$ yields a unique function $\psi \in \thomo$ such that 
\begin{equation}
\label{eq:31}
a_\Omega (\psi, \tilde{\theta}) = l(\tilde{\theta})\;\;\mbox{ for any } \tilde{\theta} \in \thomo \, .
\end{equation}
Actually, the compatibility condition $l(1) = 0$ allows us to extend \eqref{eq:31} to $\homo$. This follows by 
applying \eqref{eq:31} with $\tilde{\theta} = \theta - ((\theta, 1))_{\mo}$, for $\theta \in \homo$. Moreover, the uniqueness of the 
solution for the problem on $\thomo$ implies the uniqueness, up to a constant, of the solution for the problem on $\homo$.
\end{proof}

\

As observed in \eqref{eq:20}, the fluid equations for $\rho$ and $\Omega$ will follow by appealing to the collision invariants associated to 
the orientation $\Omega \in \sphere$, for any $W \in (\R \Omega)^\perp$. When $W = 0$, the solutions of \eqref{eq:28} are all the 
constants, and we obtain the particle number balance \eqref{eq:ContEqu}. Consider now $W \in (\R \Omega)^\perp \setminus \{0\}$ and 
$\psi _W$ a solution of \eqref{eq:28}. Obviously we have $\psi _W = \tilde{\psi}_W + \intv{\psi_W (v)\mo (v)}$, where $\tilde{\psi}_W$ 
is the unique solution of \eqref{eq:28} in $\thomo$. It is easily seen, thanks to \eqref{eq:ContEqu} and the linearity of \eqref{eq:balance}, 
that the balances corresponding to $\psi _W$ and $\tilde{\psi}_W$ are equivalent. Therefore for any $W \in (\R \Omega)^\perp$ it is enough to 
consider only the solution of \eqref{eq:28} in $\thomo$. From now on, for any $W \in (\R \Omega)^\perp$, we denote by $\psi _W$ the unique 
variational solution of \eqref{eq:28} verifying $\intv{\psi_W (v) \mo (v)} = 0$. The structure of the solutions 
$\psi _W, W \in (\R \Omega)^\perp\setminus \{0\}$ comes by the symmetry of the equilibrium $\mo$. Analyzing the rotations leaving 
invariant the vector $\Omega$, we prove as in \cite{BosCar15} the following result.
\begin{proposition}
\label{Struct1}
Consider $W \in (\R \Omega)^\perp\setminus \{0\}$. For any orthogonal transformation $\O$ of $\R^d$ leaving $\Omega$ invariant, that is $\O \Omega = \Omega$, we have
\[
\psi _W (\O v) = \psi _{^t \O W } (v),\;\;v \in \R^d \, ,
\]
where $^t \O$ denotes the transpose of the matrix $\O$.
\end{proposition}
\begin{proof}
First of all notice that $^t \O W \in (\R \Omega) ^\perp \setminus \{0\}$. We know that $\psi _W$ is the minimum point of the functional
\[
J_W (z) = \frac{\sigma}{2}\intv{|\nabla _v z|^2 \mo (v)} - \intv{(v \cdot W) z(v) \mo (v)},z\in \thomo \, .
\]
It is easily seen that, for any orthogonal transformation $\O$ of $\R^d$ leaving the orientation $\Omega$ invariant, and any function $z \in \thomo$, we have, by defining $z_{\O}: = z\circ \O \in \thomo$
\[
\mo \circ \O = \mo,\;\;\nabla  z_\O = \;^t \O (\nabla z)\circ \O \, .
\]
Moreover, we obtain with the change of variables $v' = \O v$ and using that $M_\Omega ( v) = M_\Omega ( v')$:
\begin{align*}
J_{^t\O W} (z_{\O}) & = \frac{\sigma}{2}\intv{|^t \O (\nabla  z) (\O v)|^2 \mo (v)} - \intv{(v \cdot \;^t \O W) z (\O v)\mo (v)} \\
& = \frac{\sigma}{2}\intvp{|(\nabla  z) (\vp)|^2 \mo (\vp)} - \intvp{(\vp \cdot  W) z (\vp)\mo (\vp)} \\
& = J_W(z).
\end{align*}
Finally, we deduce that
\[
\psi _W \circ \O\in \thomo,\;\;
J_{^t \O W} (\psi _W \circ \O) = J_W (\psi _W ) \leq J_W(z \circ \;^t\O) = J_{^t \O W} (z) \, ,
\]
for any $z \in \thomo$, implying that $\psi _W \circ \O = \psi _{^t \O W}$. 
\end{proof}
The computation of the collision invariants $\{\psi _W\;:\; W \in (\R \Omega) ^\perp \setminus \{0\}\}$ 
can be reduced to the computation of one scalar function. For any orthonormal basis 
$\{E_1,\ldots,E_{d-1}\}$ of $(\R \Omega) ^\perp$ we define the vector field 
$F = \sumi \psi _{E_i} E_i$. This vector field does not depend upon the basis 
$\{E_1,\ldots,E_{d-1}\}$ and has the following properties, see \cite{BosCar15}.
\begin{lemma}
\label{InvField}
The vector field $F$ does not depend on the orthonormal basis $\{ E_1, \ldots, E_{d-1}\}$ of $(\R\Omega)^\perp$ and for any orthogonal transformation $\O$ of $\R^d$, preserving $\Omega$, we have $F \circ \O = \O F$. There is a function $\chi$ such that 
\[
F(v) = \chi \left ( \frac{v \cdot \Omega}{|v|}, |v| \right ) \frac{(I_d - \Omega \otimes \Omega)(v)}{\sqrt{|v|^2 - (v \cdot \Omega)^2}},\;\;v\in \R ^d  \setminus (\R \Omega) \, ,
\]
and thus, for any $i \in \{1, ..., d-1\}$, we have
\begin{equation}\label{eq:vfield}
\psi _{E_i} (v) = F(v)\cdot E_i = \chi \left ( \frac{v \cdot \Omega}{|v|}, |v| \right )\frac{v\cdot E_i}{\sqrt{|v|^2 - (v \cdot \Omega)^2}},\;\;v\in \R ^d  \setminus (\R \Omega) \, .
\end{equation}
\end{lemma}
\begin{proof}
Let $\{F_1,\ldots, F_{d-1}\}$ be another orthonormal basis  of $(\R \Omega)^\perp$. The following  identities hold
\[
E_1 \otimes E_1 +\ldots +E_{d-1} \otimes E_{d-1} + \Omega \otimes \Omega = I_d,\;\;F_1 \otimes F_1 + \ldots+F_{d-1} \otimes F_{d-1} + \Omega \otimes \Omega = I_d \, ,
\]
and therefore 
\begin{align*}
\sumi \psi _{E_i} E_i & = \sumi \psi _{\sumj (E_i \cdot F_j)F_j} E_i \\
& = \sumi \sumj (E_i \cdot F_j) \psi _{F_j} E_i \\
& = \sumj \psi _{F_j} \sumi (E_i \cdot F_j) E_i  \\
& = \sumj \psi _{F_j} F_j \, .
\end{align*}
For any orthogonal transformation of $\R^d$ such that $\O \Omega = \Omega$ we obtain thanks to Proposition~\ref{Struct1}
\[
F\circ \O  = \sum _{i = 1} ^{d-1} ( \psi _{E_i} \circ \O ) \;E_i  = \sum _{i = 1} ^{d-1}\psi _{^t \O E_i} \;E_i = \O \sum _{i = 1} ^{d-1}\psi _{^t \O E_i } \;^t \O E_i = \O F \, ,
\]
where, the last equality holds true since $\{^t \O E_1,\ldots,\;^t\O E_{d-1}\}$ is an orthonormal basis of $(\R \Omega) ^\perp$. Let $v \in \R^d \setminus (\R \Omega)$ and consider
\[
E ( v) = \frac{(Id- \Omega \otimes \Omega)v}{\sqrt{|v|^2 - (\Omega \cdot v)^2}} \, .
\]
Notice that $E \cdot \Omega = 0, |E| = 1$. When $d = 2$, since the vector  $F(v)$ is orthogonal to $\Omega$, there exists a function $\Lambda = \Lambda (v)$ such that 
\[
F(v) = \Lambda (v) E = \Lambda (v) \frac{(I_2- \Omega \otimes \Omega)v}{\sqrt{|v|^2 - (\Omega \cdot v)^2}},\;\;v \in \R^2 \setminus ( \R \Omega) \, .
\]
If $d \geq 3$, let us denote by $^\bot E$, any unitary vector orthogonal to $E$ and $\Omega$. Introducing the orthogonal matrix $\O = I_d - 2 \;^\bot E \otimes \;^\bot E$ (which leaves $\Omega$ invariant), we obtain $F \circ \O = \O F$. Observe that
\[
0 = \;^\bot E \cdot E = \;^\bot E\cdot \frac{v - (v \cdot \Omega)\Omega}{\sqrt{|v|^2 - (v \cdot \Omega) ^2}} = \frac{\;^\bot E\cdot v}{\sqrt{|v|^2 - (v \cdot \Omega) ^2}},\;\;\O v = v \, ,
\]
and thus 
\[
F(v) = F (\O v ) = \O F(v) = (I_d - 2 \;^\bot E \otimes \;^\bot E) F(v) = F(v) - 2 ( \;^\bot E \cdot F(v)) \;^\bot E \, ,
\]
from which it follows that $^\bot E \cdot F(v) = 0$, for any vector $^ \bot E$ orthogonal to $E$ and $\Omega$. Hence, there exists a function $\Lambda (v)$ such that 
\[
F(v) = \Lambda (v) E (v) = \Lambda (v) \frac{(I_d - \Omega \otimes \Omega)v}{\sqrt{|v|^2 - (v \cdot \Omega)^2}},\;\;v \in \R ^d \setminus (\R \Omega) \, .
\]
We will show that $\Lambda (v)$ depends only on $v\cdot \Omega/|v|$ and $|v|$. Indeed, for any $d\geq 2$, and any orthogonal transformation $\O$ such that $\O \Omega = \Omega$ we have $F(\O v) = \O F(v)$, $E ( \O v) = \O E ( v)$ because
\[
& (I_d - \Omega \otimes \Omega) \O v = \O v - (\Omega \cdot \O v ) \Omega = \O (I_d - \Omega \otimes \Omega) v,
\]
\[
\sqrt{|\O v |^2 - (\O v \cdot \Omega) ^2 } = \sqrt{|v|^2 - (v \cdot \Omega) ^2} \, ,
\]
implying that $\Lambda (\O v) = \Lambda (v)$, for any $v \in \R^d \setminus (\R \Omega)$. 
We are done if we prove that $\Lambda (v) = \Lambda (\vp)$ for any $v,\vp \in \R^d \setminus (\R \Omega)$ such that $v\cdot \Omega /|v| = \vp \cdot \Omega /|\vp|,|v| = |\vp|, v \neq \vp$. It is enough to consider the rotation $\O$ such that 
\[
\O E = E{\;^\prime},\;\;(\O - I_d)_{\mathrm{span}\{E, E^{\;\prime}\}^\perp} = 0,\;\;E = \frac{(I_d - \Omega \otimes \Omega)v}{\sqrt{|v|^2 - (v \cdot \Omega)^2}},\;E^{\;\prime} = \frac{(I_d - \Omega \otimes \Omega)\vp}{\sqrt{|\vp|^2 - (\vp \cdot \Omega)^2}} \, .
\]
The equality $\O E = E ^{\;\prime}$ implies that $\O v = \vp$ and therefore $\Lambda (\vp) = \Lambda (\O v) = \Lambda (v)$, showing that there exists a function $\chi$ such that $\Lambda (v) = \chi (v \cdot \Omega/|v|, |v|), v \in \R^d \setminus \{0\}$. 
\end{proof}

\

In the last part of this section we concentrate on the elliptic problem satisfied by the function 
$(c,r) \to \chi (c,r)$ introduced in Lemma \ref{InvField}. Even if $\psi _{E_i}$ are eventually singular on $\R\Omega$, it will be no difficulty to define a Hilbert space on which solving for the profile $\chi$.
We again proceed using the minimization
of quadratic functionals.
\begin{proposition}
\label{EllipticChi}
The function $\chi$ constructed in Lemma \ref{InvField} solves the problem
\begin{align}
- \sigma \partial _c \{ r ^{d-3} (1 - c^2) ^{\frac{d-1}{2}} e(c,r) \partial _c \chi \} & - \sigma \partial _r \{ r ^{d-1} (1 - c^2 ) ^{\frac{d-3}{2}} e(c,r) \partial _r \chi \} + \sigma (d-2) r ^{d-3} (1 - c^2 ) ^{\frac{d-5}{2}} e \chi \nonumber\\
& = r^d (1- c^2 ) ^{\frac{d-2}{2}} e(c,r),\;\;(c,r) \in ]-1,+1[\;\times \;] 0, \infty[ \, . \label{eq:36}
\end{align}
\end{proposition}
\begin{proof}
For any $i \in \{1,\ldots,d-1\}$, let us consider $\psi _{E_i,h} (v) = h ( v \cdot \Omega /|v|, |v|) \frac{v \cdot E_i}{\sqrt{|v|^2 - ( v \cdot \Omega )^2}}$ where $v \to h ( v \cdot \Omega /|v|, |v|)$ is a function such that $\psi _{E_i,h} \in \homo$. Observe that if $h = \chi$, then $\psi _{E_i,h}$ coincides with $\psi _{E_i}$. Note that generally $\psi _{E_i,h}$ are not collision invariants, but perturbations of them, corresponding to profiles $h$. In this way, the minimization problem \eqref{eq:33} will lead to a minimization problem on $h$, whose solution will be $\chi$. Notice that once that $\psi _{E_i,h} \in \homo$, then $\intv{\psi _{E_i,h}\mo (v)} = 0$, saying that $\psi _{E_i,h} \in \thomo$. We know that $\psi _{E_i}$ is the minimum point of $J_{E_i}$ on $\thomo$ and therefore
\begin{equation}
\label{eq:33}
J_{E_i} (\psi _{E_i}) \leq J_{E_i}(\psi _{E_i,h}) \, .
\end{equation}
A straightforward computation shows that 
\begin{align*}
\nabla _v \psi _{E_i,h} & = \frac{v \cdot E_i}{\sqrt{|v|^2 - ( v \cdot \Omega )^2}}\left [\partial _c h \frac{I_d - \frac{v \otimes v}{|v|^2}}{|v|} \Omega + \partial _r h \frac{v}{|v|}    \right ]\\
& + h \left ( \frac{v \cdot \Omega}{|v|}, |v|\right ) \left [ I_d - \frac{(v - (v \cdot \Omega) \Omega ) \otimes v }{|v|^2 - (v \cdot \Omega)^2} \right ]\frac{E_i}{\sqrt{|v|^2 - ( v \cdot \Omega )^2}} \, ,
\end{align*}
and 
\begin{align*}
|\nabla _v \psi _{E_i,h}|^2 & = \frac{(v \cdot E_i )^2}{|v|^4} (\partial _c h )^2 + \frac{(v \cdot E_i )^2}{|v|^2 - (v \cdot \Omega) ^2} (\partial _r h )^2 
 + \frac{|v|^2 - (v \cdot \Omega) ^2 - (v \cdot E_i )^2 }{( |v|^2 - ( v \cdot \Omega )^2)^2} h ^ 2\left ( \frac{v \cdot \Omega}{|v|}, |v|\right ) \, .
\end{align*}
The condition $\psi _{E_i,h} \in \homo$ writes
\[
\intv{(\psi _{E_i,h})^2 \mo (v)} < \infty,\;\;\intv{|\nabla _v \psi _{E_i,h}|^2 \mo (v)} <  \infty \, ,
\]
which is equivalent, thanks to the Poincar\'e inequality \eqref{eq:Poincare} to
\[
\intv{|\nabla _v \psi _{E_i,h}|^2 \mo (v)} <  \infty.
\]
Based on formula \eqref{eq:34}, we have
\begin{align*}
\intv{|\nabla _v \psi _{E_i,h}|^2 \mo (v)} & = \intv{\left [\frac{|v|^2 - (v \cdot \Omega) ^2 }{(d-1) |v|^4} (\partial _c h )^2 + \frac{(\partial _r h )^2}{d-1} + \frac{(d-2)h^2}{(d-1)(|v|^2 - (v \cdot \Omega) ^2 )}   \right ] \mo}\\
& = \frac{ \int _{\R_+} r ^{d-1} \int _{-1} ^{+1}\left [\frac{1-c^2}{r^2} (\partial _c h)^2 + (\partial _r h )^2 + \frac{(d-2)h^2}{r^2 (1-c^2)}  \right ]e(c,r) (1- c^2) ^{\frac{d-3}{2}} \d c\d r}{(d-1) \int _{\R_+} r ^{d-1} \int _{-1} ^{+1} e(c,r) (1- c^2) ^{\frac{d-3}{2}} \d c\d r \}}
\end{align*}
and therefore we consider the Hilbert space $H_d = \{h :\; ]-1, +1[ \times ] 0, \infty[ \to \R, \|h \|_d ^2 < \infty \}$, 
endowed with the scalar product
\[
(g,h)_d = \int_{\R_+} r ^{d-1} \int _{-1} ^{+1}\left [\frac{1-c^2}{r^2} \partial _c g \partial _c h + \partial _r g \partial _r h + \frac{(d-2)gh}{r^2 (1- c^2)}   \right ]e(c,r) (1- c^2) ^{\frac{d-3}{2}} \d c\d r
\]
for $g$ and $h$ in $H_d$ and the norm given by
\[
\| h \|_{d} = \sqrt{ ( h, h)_d} \, .
\]
The expression $J_{E_i}(\psi _{E_i,h})$ writes as functional of $h$
\begin{align*}
J_{E_i}(\psi _{E_i,h}) & = \frac{\sigma}{2}\intv{|\nabla _v \psi _{E_i,h}|^2 \mo (v)} - \intv{\frac{(v \cdot E_i)^2 h\left ( \frac{v \cdot \Omega}{|v|}, |v|\right )}{\sqrt{|v|^2 - (v \cdot \Omega)^2}}\mo (v)}\\
& = \frac{\sigma}{2} \intv{|\nabla _v \psi _{E_i,h}|^2 \mo (v)} - \intv{\frac{h\left ( \frac{v \cdot \Omega}{|v|}, |v|\right )\sqrt{|v|^2 - (v \cdot \Omega)^2}}{d-1}\mo (v)}\\
& = \frac{J(h)}{(d-1) \int_{\R_+}r^{d-1} \int _{-1} ^{+1} e(c,r) ( 1 - c^2) ^{\frac{d-3}{2}}\d c \d r } \, ,
\end{align*}
where 
\begin{align*}
J(h) & = \frac{\sigma}{2} \int _{\R_+} r ^{d-1} \int _{-1} ^{+1} \left [\frac{1-c^2}{r^2} ( \partial _c h)^2 + (\partial _r h)^2 + \frac{(d-2)h^2}{r^2 (1- c^2)} \right ] e(c,r) (1-c^2) ^{\frac{d-3}{2}}\d c \d r \\
& - \int _{\R_+} r ^{d-1} \int _{-1} ^{+1} h(c,r) r \sqrt{1-c^2} e(c,r) (1-c^2) ^{\frac{d-3}{2}}\d c \d r \, .
\end{align*}
Coming back to \eqref{eq:33} and using \eqref{eq:vfield}, we deduce that
\[
\chi \in H_d\;\;\mbox{ and } \;\;J(\chi) \leq J(h)\;\;\mbox{ for any }\;\; h \in H_d \, .
\]
Therefore, by the Lax-Milgram lemma, we deduce that $\chi$ solves the problem \eqref{eq:36}. 
\end{proof}

\section{Hydrodynamic equations}
\label{Hydro}
After identifying the collision invariants, we determine the fluid equations satisfied by the macroscopic 
quantities entering the dominant 
particle density $f(t,x,v) = \rho (t,x) M_{\Omega (t,x)} (v)$. As seen before the balance for 
the particle density follows thanks to the collision invariant $\psi = 1$.
The other balances follow by appealing to the vector field $F$ (cf. Lemma \ref{InvField})
and the details are given in Sect. \ref{sec:hydro}. 
\subsection{Proof of Theorem \ref{MainResult1}}\label{sec:hydro}
Applying \eqref{eq:20} with $\psi = 1$ leads to \eqref{eq:ContEqu}. For any $(t,x) \in \R_+ \times \R^d$ we consider the vector field
\[
v \to F(t,x,v) = \chi \left (\frac{v \cdot \Omega (t,x)}{|v|}, |v|   \right ) \frac{(I_d - \Omega (t,x) \otimes \Omega (t,x))v}{\sqrt{|v|^2 - (v \cdot \Omega (t,x))^2}} \, .
\]
By the definition of $F(t,x,\cdot)$, we know that
\[
\calL _{f(t,x,\cdot)} ^\star F(t,x,\cdot) = 0,\;\;(t,x) \in \R_+ \times \R^d \, ,
\]
and therefore \eqref{eq:20} implies
\begin{equation}
\label{eq:40}
\intv{\partial _t \fz \;F(t,x,v)} + \intv{v \cdot \nabla _x \fz  \;F(t,x,v) } = 0,\;\;(t,x) \in \R_+ \times \R^d \, .
\end{equation}
It remains to compute $\intv{\partial _t \fz F(t,x,v)}$ and $\intv{v \cdot \nabla _x \fz  F(t,x,v) }$ 
in terms of $\rho (t,x)$ and $\Omega (t,x)$. By a direct computation we obtain
\[
\partial _t \fz = \partial _t \rho \mo + \rho \frac{\mo }{\sigma} (v - \Omega) \cdot \partial _t \Omega \, ,
\]
implying, thanks to the equalities $\partial _t \Omega \cdot \Omega = 0$ and $v \cdot \partial_t \Omega = P_f v \cdot \partial_t \Omega$,
\begin{align}
\label{eq:41}
& \intv{\partial _t \fz \;F}  = \intv{\left ( \partial _t \rho + \frac{\rho }{\sigma} (v - \Omega) \cdot \partial _t \Omega \right ) \chi \left ( \frac{v \cdot \Omega}{|v|}, |v|\right ) \mo (v) \frac{P_f v}{|P_f v |}}\\
& = \partial _t \rho \intv{\chi \left ( \frac{v \cdot \Omega}{|v|}, |v|\right )\mo (v) \frac{P_f v}{|P_f v |}} + \frac{\rho}{\sigma} \intv{\chi \left ( \frac{v \cdot \Omega}{|v|}, |v|\right )\mo (v) \frac{P_f v \otimes P_f v}{|P_f v|}}\;\partial _t \Omega \, . \nonumber 
\end{align}
It is an easy exercise to show that the integral $\intv{\chi \left ( \frac{v \cdot \Omega}{|v|}, |v|\right )\mo (v) \frac{P_f v}{|P_f v |}}$ 
vanishes and that the following relationship holds
\begin{align}
& \intv{\chi \left ( \frac{v \cdot \Omega}{|v|}, |v|\right )\mo (v) \frac{P_f v \otimes P_f v}{|P_f v|}} \nonumber \\
& \hspace{4cm} = \intv{\chi \left ( \frac{v \cdot \Omega}{|v|}, |v|\right )\frac{\sqrt{|v|^2 - (v \cdot \Omega)^2}}{d-1}\mo (v)}(I_d - \Omega \otimes \Omega) \, . \nonumber 
\end{align}
Therefore, by taking into account that $\partial _t \Omega \cdot \Omega = 0$, the equality \eqref{eq:41} becomes
\begin{equation}
\label{eq:42}
\intv{\partial _t \fz \;F}  = \tilde{c}_1 \frac{\rho}{\sigma} \partial _t \Omega,\;\;\tilde{c}_1 = \intv{\chi \left ( \frac{v \cdot \Omega}{|v|}, |v|\right )\frac{\sqrt{|v|^2 - (v \cdot \Omega)^2}}{d-1}\mo (v)} \, .
\end{equation}
Similarly we write (for any smooth vector field $\xi (x)$, the notation $\partial _x \xi$ stands for the Jacobian matrix of $\xi$, i.e. $( \partial_x \xi)_{ i, j} = \partial_{x_j} \xi_i$)
\[
v \cdot \nabla _x f = (v \cdot \nabla _x \rho ) \mo + \frac{\rho}{\sigma} v \cdot ( \;^t \partial _x \Omega (v-\Omega)) \mo \, , 
\]
implying 
\begin{align}
\label{eq:43}
\intv{(v \cdot \nabla _x f ) \;F} & = \intv{\chi \left ( \frac{v \cdot \Omega}{|v|}, |v|\right )\mo (v) \frac{P_f v \otimes P_f v}{|P_f v|} } \, \nabla _x \rho \\
& + \frac{\rho}{\sigma} \intv{(v \cdot \Omega) \chi \left ( \frac{v \cdot \Omega}{|v|}, |v|\right )\mo (v) \frac{P_f v \otimes P_f v}{|P_f v|} }\; ( \partial _x \Omega) \, \Omega \nonumber\\
& = \tilde{c}_1 (I_d - \Omega \otimes \Omega) \nabla _x \rho + \frac{\rho}{\sigma}\tilde{c}_2 \partial _x \Omega \;\Omega \, , \nonumber
\end{align}
where 
\[
\tilde{c}_2 = \intv{(v \cdot \Omega) \chi \left ( \frac{v \cdot \Omega}{|v|}, |v|\right )\frac{\sqrt{|v|^2 - (v \cdot \Omega)^2}}{d-1}\mo (v)} \, .
\]
Notice that in the above computations we have used $(^t \partial _x \Omega) \Omega = 0$ and
\[
\intv{(v \cdot E_i) (v \cdot E_j) (v \cdot E_k)  \chi \left ( \frac{v \cdot \Omega}{|v|}, |v|\right )\mo (v)} = 0 \, ,
\]
for any $i,j,k \in \{1,\ldots, d-1\}$. Combining \eqref{eq:40}, \eqref{eq:42} and \eqref{eq:43} yields
\[
\tilde{c}_1 \frac{\rho}{\sigma} \partial _t \Omega + \frac{\rho}{\sigma} \tilde{c}_2 \;\partial _x \Omega \;\Omega + \tilde{c_1} (I_d - \Omega \otimes \Omega) \nabla _x \rho = 0 \, ,
\]
or equivalently 
\begin{equation}
\label{eq:OmegaEvol}
\partial _t \Omega + c_2 \;\partial _x \Omega \;\Omega + \sigma (I_d - \Omega \otimes \Omega)\frac{\nabla _x \rho}{\rho} = 0 \, ,
\end{equation}
where 
\begin{align*}
c_2 & = \frac{\tilde{c}_2}{\tilde{c}_1} = \frac{\intv{\;(v \cdot \Omega) \, \chi \left ( \frac{v \cdot \Omega}{|v|}, |v|\right )\sqrt{|v|^2 - (v \cdot \Omega)^2}\;\mo (v)}}{\intv{ \chi \left ( \frac{v \cdot \Omega}{|v|}, |v|\right )\sqrt{|v|^2 - (v \cdot \Omega)^2}\;\mo (v)}}\\
& = \frac{\int_{\R_+}r^{d+1} \intth{\cos \theta \, \chi (\cos \theta, r) \, e(\cos \theta, r) \sin ^{d-1} \theta}\d r}{\int_{\R_+}r^{d} \intth{\chi (\cos \theta, r) \, e(\cos \theta, r) \sin ^{d-1} \theta}\d r} \, .
\end{align*}

\subsection{Properties of the hydrodynamic equations}
Let us start by noticing that the system \eqref{eq:bas_rho}-\eqref{eq:bas_omega} is hyperbolic as a consequence of 
Theorem 4.1 in \cite{F12}. On the other hand, the orientation balance equation \eqref{eq:OmegaEvol} propagates the 
constraint $|\Omega| = 1$. Indeed, let $\Omega = \Omega (t,x)$ be a smooth solution of \eqref{eq:OmegaEvol}, 
satisfying $|\Omega (0, \cdot)| = 1$. Multiplying by $\Omega (t,x)$ we obtain
\[
\frac{1}{2} \partial _t |\Omega |^2 + \frac{c_2}{2} ( \Omega (t,x) \cdot \nabla _x ) |\Omega |^2 = 0 \, ,
\]
implying that $|\Omega|$ is constant along the characteristics of the vector field $c_2 \Omega (t,x) \cdot \nabla _x$ 
and thus $|\Omega (t,\cdot)| = |\Omega (0,\cdot)| = 1$, for all $t \geq 0$.

The rescaled equation \eqref{eq:main} can be considered as an intermediate model between the equations introduced 
in \cite{BD14} and \cite{BosCar15} when there are no phase transitions. 
In \cite{BosCar15}, the authors considered a strong relaxation towards the 'terminal speed' (or cruise speed). Whereas in \cite{BD14} 
  the authors do not impose a penalization on the self-propelled/friction term. 
Our result could be applied to obtain the results in \cite{BosCar15} without resorting to measures supported on the sphere 
by doing a double passage to the limit. First, passing to the limit $\eps \rightarrow 0$ in \eqref{eq:main}, taking $V = V_{\alpha, \beta}$, we obtain
\eqref{eq:bas_rho}-\eqref{eq:bas_omega}. Afterwards, we rescale $V$ to $\lambda \widetilde{V}$ in the system \eqref{eq:bas_rho}-\eqref{eq:bas_omega}
and study the limit when $\lambda \rightarrow \infty$. This amounts to study the behavior of the coefficients 
\[
 c_{ 1, \lambda} =  \frac{\int_{\R_+}r^d \intth{\cos \theta \, e_\lambda (\cos \theta, r) \sin ^{d-2} \theta}\d r}{\int_{\R_+}r^{d-1} \intth{e_\lambda (\cos \theta, r) \sin ^{d-2} \theta}\d r} \, ,
\]
and
\[
c_{2, \lambda} = \frac{\int_{\R_+}r^{d+1} \intth{\cos \theta \, \chi_\lambda (\cos \theta, r) \, e_\lambda (\cos \theta, r) \sin ^{d-1} \theta}\d r}{\int_{\R_+}r^{d} \intth{\chi_\lambda (\cos \theta, r) \, e_\lambda (\cos \theta, r) \sin ^{d-1} \theta}\d r} \, ,
\]
when $\lambda \rightarrow \infty$, where the function $\chi_\lambda$ solves the elliptic problem
\begin{align*}
& - \sigma \partial _c \left [r^{d-3} (1 - c^2) ^{\frac{d-1}{2}} e_\lambda(c,r) \partial _c \chi_\lambda   \right ] \\
& \qquad - \sigma \partial_r \left [r^{d-1}(1 - c^2 ) ^{\frac{d-3}{2}} e_\lambda(c,r) \partial_r \chi_\lambda  \right ] + \sigma (d-2) r ^{d-3} (1- c^2) ^{\frac{d-5}{2}} e_\lambda \chi_\lambda \nonumber \\
& \qquad \qquad = r^d (1-c^2) ^{\frac{d-2}{2}}e_\lambda(c,r) \, ,
\end{align*}
and $e_\lambda (c,r) = \exp(rc/\sigma) \exp(- (r^2 + 1)/(2\sigma) - \lambda \widetilde{V} (r)/\sigma)$. 

In order to analyse the asymptotic behavior of $c_{1, \lambda}$ we introduce the following result.

\begin{lemma}\label{lem:asymptotic}
 Let $\varphi \in C^2 ( (0, \infty); \R)$ and $g \in C^0 ( (0, \infty) \times ( 0, \pi); \R )$.
 Let us assume that 
 \begin{itemize}
  \item[i)] $\int_{\R_+} \int_0^\pi \exp ( \lambda \varphi ( r)) | g ( r, \theta)| \d \theta \d r < \infty$ ,
  \item[ii)] The function $\varphi$ has a unique global maximum at an interior point $r_0$ ,
  \item[iii)] $\int_0^\pi g ( r_0, \theta) \d \theta \neq 0$ .
 \end{itemize}
 Then the function $G( \lambda)$ defined as 
 \[
  G( \lambda) = \int_{\R_+} \int_0^\pi \exp ( \lambda \varphi ( r)) g ( r, \theta) \d \theta \d r \, ,
 \]
 has the the following asymptotic behavior
 \[
  G( \lambda) \sim \sqrt{ \frac{ 2 \pi }{ | \varphi'' ( r_0)|}} \frac{\exp ( \lambda \varphi ( r_0))}{ \sqrt{ \lambda}} \int_0^\pi g( r_0, \theta) \d \theta \, , 
 \]
 as $\lambda \rightarrow \infty$.
\end{lemma}
The proof of this result is a direct application of the Laplace method, see for instance \cite{Mil06}. 
As an immediate consequence of Lemma \ref{lem:asymptotic} we obtain 

\[
 \lim_{ \lambda \rightarrow \infty} c_{ 1, \lambda} = \frac{ r_0 \int_0^\pi \cos \theta \exp ( r_0 \cos \theta / \sigma) \sin^{d-2} \theta \d \theta}{ \int_0^\pi \exp ( r_0 \cos \theta / \sigma) \sin^{d-2} \theta \d \theta} \, ,
\]
where $r_0$ is the minimum of the potential function $V_{\alpha, \beta} ( r)$. Let us note that the asymptotic study of
the coefficient $c_{1 , \lambda}$ when $\lambda \rightarrow \infty$ can also be performed using Lemma 
\ref{lem:asymptotic} for more general potentials than 
$V_{\alpha, \beta} ( r)$. In particular, we could also consider smooth potentials $V ( |v|)$ having a unique 
global minimum $r_0$ such that $V' (r) < 0$, for $0 < r < r_0$, and $V' ( r) > 0$, for $ r > r_0$.  
On the other hand, the asymptotic study of $c_{2,\lambda}$ could be addressed
following similar techniques as in \cite{F12}, however, we do not dwell upon this matter here and leave it for a future work.
\subsection*{Acknowledgments}
PAS and PD acknowledge support by the Engineering and Physical Sciences Research Council (EPSRC) under grants no. EP/M006883/1 and EP/P013651/1. This work was accomplished during the visit of MB to Imperial College London with an ICL-CNRS Fellowship.
JAC was partially supported by the EPSRC grant number EP/P031587/1. PD also acknowledges support by the Royal Society and the Wolfson Foundation through a Royal Society Wolfson Research Merit Award no. WM130048 and by the National Science Foundation (NSF) under grant no. RNMS11-07444 (KI-Net). PD is on leave from CNRS, Institut de Math\'ematiques de Toulouse, France. PAS and PD would like to thank the Isaac Newton Institute for Mathematical Sciences for support and hospitality during the programme ``Growth form and self-organisation'' when part of this work was carried over which was supported by: EPSRC grant numbers EP/K032208/1 and EP/R014604/1. \\

Data statement: no new data were collected in the course of this research.

\end{document}